\documentclass[11pt]{article}

\usepackage{fullpage}
\usepackage{amsthm}
\usepackage{amssymb}
\usepackage{amsmath}
\usepackage{graphicx}
\usepackage{tikz}
\usepackage{mathrsfs}
\usepackage{float}
\usepackage{makecell}
\usepackage{amscd}
\usepackage{multirow}
\usepackage{hhline}
\usepackage{verbatim}

\newcommand{\Z}{\mathbb{Z}}
\newcommand{\N}{\mathbb{N}}

\newcommand{\Q}{\mathbb{Q}}
\newcommand{\Y}{\mathcal{Y}}

\newcommand{\X}{\mathcal{X}}

\renewcommand{\Im}{\text{Im}}

\DeclareMathOperator{\coker}{coker}
\DeclareMathOperator{\supp}{supp}
\DeclareMathOperator{\rank}{rank}

\usepackage{relsize}
\newtheorem{theorem}{Theorem}
\newtheorem{lemma}{Lemma}
\newtheorem*{unnumtheorem}{Theorem}

\newtheorem{remark}{Remark}

\newtheorem{corollary}[lemma]{Corollary}
\newtheorem{claim}[lemma]{Claim}
\theoremstyle{definition}

\newtheorem{definition}{Definition}

\title{The integer homology threshold in $Y_d(n, p)$} 
\author{Andrew Newman\thanks{The Ohio State University \& Technische Universit\"{a}t Berlin}, Elliot Paquette\thanks{The Ohio State University}}
\date{\today}

\begin{document}
\maketitle
\abstract
We prove that in the $d$-dimensional Linial--Meshulam stochastic process the $(d - 1)$st homology group with integer coefficients vanishes exactly when the final isolated $(d - 1)$-dimensional face is covered by a top-dimensional face. This generalizes the $d = 2$ case proved recently by \L uczak and Peled and establishes that $p = \frac{d \log n}{n}$ is the sharp threshold for homology with integer coefficients to vanish in $Y_d(n, p),$ answering a 2003 question of Linial and Meshulam.

\section{Statement of the Result}

Here we consider the stochastic process version of the Linial--Meshulam random simplicial complex model. Recall that the Linial--Meshulam random simplicial complex model (introduced in \cite{LM}), denoted $Y_d(n, p)$ for $d$ a fixed dimension, $n \in \N$, and $p = p(n) \in [0, 1]$, is the probability space on $d$-dimensional simplicial complexes with complete $(d - 1)$-skeleton generated by including each possible $d$-dimensional face independently with probability $p$.  Accordingly the (discrete-time) stochastic process version of $Y_d(n, p)$, which we denote here as $\Y_d(n)$, following \cite{LP2}, is a Markov process $Y_d(n, 0) \subseteq Y_d(n, 1) \subseteq \cdots \subseteq Y_d(n, \binom{n}{d + 1})$ where $Y_d(n, 0)$ is the complete $(d - 1)$-complex on $n$ vertices and $Y_d(n, k)$ is generated by adding a $d$-dimensional face to $Y_d(n, k - 1)$ chosen uniformly at random from among all $d$-dimensional faces not included in $Y_d(n, k - 1)$.

For a topological property $P$ and a single instance of $\Y_d(n)$, the \emph{hitting time} for property $P$ is defined to be the minimal $m$ so that $Y_d(n, m)$ satisfies property $P$. A statement $S$ about $\Y_d(n)$ is said to hold \emph{with high probability} if the probability $S$ holds tends to $1$ as $n \to \infty.$

Our main new result on $\Y_d(n)$ is the following theorem; following tradition, we call a $(d-1)$-dimensional face \emph{isolated} if it is not covered by any $d$-dimensional face.
\begin{theorem}\label{mainresult}
Fix $d \geq 2$. With high probability the $(d - 1)$st homology group of $Y = \Y_d(n)$ with integer coefficients vanishes exactly when the last isolated $(d - 1)$-dimensional face is covered by a $d$-dimensional face. That is the hitting time for the property that no $(d - 1)$-dimensional face of $Y$ is isolated exactly coincides with the hitting time for the property that $H_{d - 1}(Y; \Z) = 0$.
\end{theorem}
\noindent In fact we prove a slightly stronger result in Corollary \ref{isolatedfaces}, which shows that slightly before the final isolated $(d - 1)$-dimensional face is covered, the $(d - 1)$st homology group is a free abelian group with rank given by the number of isolated faces.

The $d = 2$ case of Theorem \ref{mainresult} was previously established by \L uczak and Peled \cite{LP2}. Moreover, the $d = 1$ is the classic result of Bollobas and Thomason \cite{BT} that the stochastic random graph becomes connected at the exact moment its last isolated vertex is covered by an edge.

\section{Background}
The Linial--Meshulam model is a higher-dimensional generalization of the Erd\H{o}s--R\'{e}nyi random graph, and one of the most fundamental results in random graph theory is the following theorem of Erd\H{o}s and R\'{e}nyi which establishes the connectivity threshold for $G(n, p)$. 
\begin{unnumtheorem}[\cite{ER}]
For $c < 1$ and $p = \frac{c \log n}{n}$, with high probability $G \sim G(n, p)$ is not connected, and for $c > 1$ and $p = \frac{c \log n}{n}$, with high probability $G \sim G(n, p)$ is connected.
\end{unnumtheorem}

This statement can be given a homological reformulation by observing that a graph is connected if and only if its zeroth reduced homology group is trivial.
This motivates the general definition of \emph{homological connectivity over abelian group $R$} for higher dimensional simplicial complexes.
A $d$-dimensional simplicial complex $X$ is said to be homologically connected over $R$ provided that $\tilde{H_i}(X; R) = 0$ for all $i \leq d - 1$. In $Y_d(n, p)$, all complexes have complete $(d - 1)$-skeleton so homological connectivity over $R$ of $Y \sim Y_d(n, p)$ is equivalent to $\tilde{H}_{d - 1}(Y; R) = 0$. In the case that $R= \Z,$ we will simply say the complex is \emph{homologically connected}.  Indeed, by the universal coefficient theorem, a complex $X$ is homologically connected if and only if it is homologically connected over all abelian groups $R.$

%Another, weaker generalization is \emph{homological connectivity over $R$} for an abelian group $R.$  In this case, say that a $d$-dimensional complex $X$ is homologically connected over $R$ provided that $\tilde{H_i}(X; R) = 0$ for all $i \leq d - 1$.  As before, for $Y \sim Y_d(n, p)$ this is equivalent to $\tilde{H}_{d - 1}(Y; R) = 0$. 

Generalizing the connectivity result of Erd\H{o}s and R\'{e}nyi, Linial and Meshulam prove the following.
\begin{unnumtheorem}[\cite{LM}]
For $c < 2$ and $p = \frac{c \log n}{n}$, with high probability $Y \sim Y_2(n, p)$ satisfies $H_1(Y; \Z/2\Z) \neq 0$, and for $c > 2$ and $p = \frac{c \log n}{n}$, with high probability $Y \sim Y_2(n, p)$ satisfies $H_1(Y; \Z/2\Z)  = 0$. 
\end{unnumtheorem}

For any $d$-complex $Y$, $H_{d - 1}(Y; \Z/2\Z) = 0$ implies that $H_{d - 1}(Y; \Q) = 0$. By the universal coefficient theorem, therefore, the above result of Linial and Meshulam implies that $H_1(Y;\Z)$ is finite for $Y \sim Y_2(n, \frac{c \log n}{n})$ and $c > 2$. However, the result \emph{does not} imply that $H_1(Y, \Z) = 0$ in this case; a priori, it may be some other finite group. Thus, unlike the case of the Erd\H{o}s--R\'{e}nyi random graph, it is not sufficient to consider only $\Z/2\Z$ coefficients to show that $H_{d - 1}(Y)$ is trivial. Generalizing the Linial--Meshulam result to higher dimensions and to other coefficients rings, Meshulam and Wallach prove the following result.
\begin{unnumtheorem}[\cite{MW}]
Fix $d \geq 1$, and let $R$ be a fixed finite abelian group. For $c < d$ and $p = \frac{c \log n}{n}$, with high probability $Y \sim Y_d(n, p)$ satisfies $H_{d - 1}(Y; R) \neq 0$, and for $c > d$ and $p = \frac{c \log n}{n}$, with high probability $Y \sim Y_d(n, p)$ satisfies $H_{d - 1}(Y; R) = 0$. 
\end{unnumtheorem}

If $Y$ is a $d$-dimensional simplicial complex with $H_{d - 1}(Y; \Z/q\Z) = 0$ for all primes $q$ then $H_{d - 1}(Y; \Z) = 0$. However, the theorem of Meshulam and Wallach does not rule out the possibility that $H_{d - 1}(Y; \Z)$ has $q$-torsion for a sequence of primes $q$ which grow with $n$. Indeed until now the question of the homological connectivity threshold (with integer coefficients) has been open for all $d \geq 3$. Previously, the best result for $d \geq 3$ about the homological connectivity threshold was the following result of Hoffman, Kahle and Paquette:
\begin{unnumtheorem}[\cite{HKP2}]
For $d \geq 2$ and $p \geq \frac{80d \log n}{n}$, with high probability $Y \sim Y_d(n, p)$ satisfies $H_{d - 1}(Y) = 0$.
\end{unnumtheorem}

For $d = 2$, the main hitting-time result of \cite{LP2} establishes that $\frac{2 \log n}{n}$ is the sharp threshold for the first homology group with integer coefficients to vanish in $Y_2(n, p)$, and our result here generalizes this hitting-time result to higher dimensions. 

We should also mention that over the field $\Z/2\Z$, the hitting-time result has been established. If one considers homology with $\Z/2\Z$ coefficients the hitting-time result was previously known in the $d = 2$ case due to Kahle and Pittel \cite{KP}, and more recently the $\Z/2\Z$ version of the hitting-time for homological connectivity was proved by Cooley et al. \cite{CdGKS} in all dimensions.  In addition a hitting-time result that for $\Q$-coefficients is proved in \cite{HKP}.
%\abstract
\section{Cocycle counting}
While the $d = 2$ case of Theorem \ref{mainresult} has already been established in \cite{LP2}, we develop a new approach based on the methods of Meshulam and Wallach \cite{MW}.  They develop the technique of
\emph{cocycle counting} to show that for any fixed finite abelian group $R$,  $H_{d - 1}(Y_d(n, p); R) = 0$ when $p = \frac{d \log n + \omega(1)}{n}.$ 
That is, rather than considering homology, they consider cohomology and bound the probability that $Y \sim Y_d(n, p)$ has a nontrivial cocycle. When $R$ is fixed and finite this may be accomplished by showing that the expected cardinality $\mathbb{E}|H^{d - 1}(Y; R)|$ tends to $0.$
%If $\partial^*_d : C^{d - 1} \rightarrow C^d$ is the coboundary operator for $Y$ and $\phi$ is a $(d - 1)$-cochain which is not a $(d - 1)$-coboundary then one considers the probability that $\phi$ is in the kernel of $\partial^*_d$. From here, one can use linearity of expectation across the finite vector space over $R$ to show that the expected size of $H^{d - 1}(Y; R)$ tends to zero. 

We will adapt this technique to work over many fields simultaneously.  Following \cite{MW}, we start by defining some useful notation.
\begin{definition}
For a $(d - 1)$-cochain $\phi$ of the simplex on $n$ vertices with coefficients in any field $R$, the \emph{weight} of $\phi$, denoted $w(\phi)$, is defined to be the minimum of the support of $\phi'$ for all $\phi'$ with $\phi - \phi'$ a coboundary, and $b(\phi)$ is defined to be the number of $d$-dimensional faces $\sigma$ in the simplex on $n$ vertices so that $\partial_d^*(\phi)(\sigma) \neq 0$.
\end{definition}
With this notation in hand, Meshulam and Wallach prove the following coisoperimetric inequality \cite[Proposition 3.1]{MW}.

\begin{lemma}[Coisoperimetric inequality]\label{coisoperimetric}
For any abelian group $R$ and any $(d - 1)$-chain of the simplex on $n$ vertices, 
$$b(\phi) \geq \dfrac{n w(\phi)}{d + 1}.$$
\end{lemma}
\noindent Note that although $b(\phi)$ and $w(\phi)$ depend on the underlying ring $R$, the coisoperimetric inequality is uniform over all abelian groups $R$.

We will now sketch the basic cocycle counting method.
Observe that if $\phi$ is a $(d - 1)$-cochain then the probability that it is a cocycle over $Y \sim Y_d(n, p)$ is $(1 - p)^{b(\phi)}$. 
Consider each equivalence class of cochains modulo coboundaries and choose a minimal-support element from each equivalence class. If $\phi$ is a cochain with support size $k$ and weight equal to $k$, then by Lemma \ref{coisoperimetric} the probability that $\phi$ is a cocycle is at most $(1 - p)^{nk/(d + 1)}$. Thus if $R$ is a fixed finite field of size $r$, we have that for $Y \sim Y_d(n, p),$
$$
\Pr(H^{d - 1}(Y; R) \neq 0) \leq \sum_{k = 1}^{\binom{n}{d}} \binom{\binom{n}{d}}{k} (r - 1)^k (1 - p)^{nk/(d + 1)}.$$

Indeed there are $\binom{\binom{n}{d}}{k}$ choices for the support of a cochain of weight $k$, and from there at $(r - 1)$ choices for the coefficient associated to each facet in the support. It follows that if $p = \frac{c \log n}{n}$ for $c > d(d + 1)$ then $H^{d - 1}(Y; R) = 0$ for $Y \sim Y_d(n, p)$. In order to improve on this, Mesulam and Wallach find a better bound than $\binom{\binom{n}{d}}{k} (r - 1)^k$ for the number of nontrivial cochains and use the coisopermetric inequality in a more subtle way.\\

In the current situation we want to show that homology with integer coefficients vanishes.
Our approach is based on the elementary observation that for any simplicial complex $X,$ if $H_{d - 1}(X; \Z/q\Z) = 0$ for all primes $q$ then $H_{d - 1}(X; \Z) = 0$. 
So, we adapt the cocycle counting method of Meshulam and Wallach to work over $\Z/q\Z$ for all primes $q$ simultaneously. 

It is worth pointing out that a direct first-moment argument alone cannot work if $q$ is very large. If we sample $Y$ from $Y_d(n, \frac{c \log n}{n})$ then the probability that $Y$ has no $d$-dimensional faces is $\exp(-\Theta(n^d \log n))$.  In this case, the dimension of $H_{d - 1}(Y; R)$ for any field $R$ is $\binom{n - 1}{d}$, so if $R = \Z/q\Z$, for $q$ a prime larger than $\exp(n^d)$, then the expected number of cocycles over $R$ is at least $\exp(\Theta(n^{2d}) - \Theta(n^d \log n)) \rightarrow \infty$. Thus is critical that we eliminate the $(r - 1)^k$ term from the cocycle counting method.\\

 To do so, rather than consider $(d-1)$-cochains, we consider $(d - 1)$-dimensional complexes and bound the probability that any such complexes support a cocycle over any prime-order finite field. We will make use of the coisoperimetric inequality (Lemma \ref{coisoperimetric}), but now in a more geometric way. We first define the following geometric analogue to $b(\phi)$. 

\begin{definition}
For a fixed field $R$ and any $(d - 1)$-dimensional subcomplex $X$ of the simplex on $n$ vertices we define $b(X, R)$ as:
$$b(X, R) := \inf \{b(\phi) : \text{$\phi$ is supported exactly on $X$ with coefficients in $R$ with $w(\phi) = |X|$} \}.$$
Note that $b(X, R)$ can be infinity but only in the situation where there are no cochains of minimum weight supported on $X$.
We also define $b(X)$ to be the infimum of $b(X, R)$ over $R = \Z/q\Z$ for all primes $q$ and $R = \Q$.
\end{definition}

Now $b(X)$ is closely related to $b(\phi)$ where $X$ is a $(d - 1)$-complex and $\phi$ is a cochain supported on $X$, but it removes everything about an underlying coefficient ring. Meshulam and Wallach also define a geometric quantity $\beta(X),$ closely related to $b(\phi)$ for $\phi$ supported on $X,$ as follows.
\begin{definition}
For a $(d - 1)$-dimensional subcomplex $X$ of the simplex on $n$ vertices we define $\beta(X)$ to be the number of $d$-dimensional faces which contain exactly one $(d - 1)$-dimensional face of $X$. 
\end{definition}
We make use of this definition too in Section \ref{largecocycles}. In outlining their proof in \cite{MW}, Meshulam and Wallach point out that the coisoperimetric inequality does not hold if $b(\phi)$ is replaced with $\beta(X)$ and that this is a major obstacle to applying their technique to prove that integer homology vanishes.  Nonetheless, it is a useful quantity because for any cochain $\phi$ with coefficients in $R$, minimally-supported on $X,$ one has $\beta(X) \leq b(X) \leq b(\phi)$. \\

So while $\beta(X)$ does not satisfy the coisoperimetric inequality, the geometric quantity $b(X)$ \emph{does} satisfy it. Indeed, due to the uniformity in $R$ in the coisoperimetric inequality (Lemma~\ref{coisoperimetric}), 
$$b(X) \geq \frac{n |X|}{d + 1}.$$

There is one potential disadvantage to using $b(X)$ in place of $b(\phi)$. If $R$ is a coefficient ring and if $\phi$ is a cochain over $R$ which is minimally supported on $X$, then it is clear from the definition of $b(\phi)$ that the probability that $\phi$ is a cocycle of $Y_d(n, p)$ is $(1 - p)^{b(\phi)}$. However, if instead we wish to bound the probability that $X$ is the support of a cocycle over any prime-order finite field, we no longer have the simple bound $(1 - p)^{b(X)}$.  Nonetheless, as we will show in Lemma \ref{probabilitybound}, this bound is true up to a lower--order correction.  

To frame Lemma \ref{probabilitybound}, we begin by introducing the following notation.
\begin{definition}
For a fixed field $R$ and any $(d - 1)$-dimensional subcomplex $X$ of the simplex on $n$ vertices, we let $z(X, R)$ denote the event that there exists a cocycle $\phi$ over $R$ with $w(\phi) = |X|$ and $\supp(\phi) = X$. We let $z(X)$ denote the event that there exists $R$ in $\{\Z/q\Z : \text{$q$ is prime and at most $\sqrt{d + 1}^{|X|}$} \} \cup \{\Q\}$ so that $z(X, R)$ holds.
\end{definition}

The choice of $\sqrt{d + 1}^{|X|}$ in the definition of $Z(X)$ comes from a bound on the size of the torsion group of the cokernel of an integral matrix, given as Claim \ref{matrixbound}. This claim is essentially Proposition 3 of \cite{Soule} who credits it to Gabber. However, we do have a sharper exponent in our bound than in \cite{Soule} ($t^{\rank(M)}$ compared to $t^{\min\{n, m\}}$). This sharper exponent is not necessary to our application here, but in the interest of keeping the proof self-contained we give a proof of Claim \ref{matrixbound}, and no additional work is require to obtain the sharper exponent.
\begin{claim}\label{matrixbound}
If $M$ is a matrix with integer entries so that the norm of every column of $M$ is at most $t$, then the torsion part of the cokernel of $M$, denoted $\coker(M)_T,$ has size at most $t^{\rank_\Q(M)}$
\end{claim}
\begin{proof}
Let $M$ be a matrix which satisfies our assumptions. First, define $N$ to be a restriction of $M$ to a maximal set of $\Q$-linearly independent columns of $M$. We have that $\coker(M)_T \leq \coker(N)_T$. Indeed, this immediate as clearly $\Im_{\Z}(N) \leq \Im_{\Z}(M)$. \\

%Now we add to $N$ exactly the standard basis vectors $e_i$ for $1 \leq i \leq m$ which are not already in the $\Q$-span of $N$. Call this matrix $N'$. By construction the columns of $N'$ are a integral basis for $\Q^m$. Thus $\coker(N')$ is a finite group and has size equal to the determinant of $N'$. By Hadamard's inequality $\det(N') \leq t^{|N|} = t^{\rank(M)}$. All that remains is to prove that $\coker(N)_T \leq \coker(N')_T$. \\

Now we want to construct a square matrix from $N$ in a canonical way. Beginning with $N$ let $i_1$ be the smallest index in $\{1, ..., m\}$ so that the standard basis vector $e_{i_1}$ is not in the $\Q$-span of $N$; add $e_{i_1}$ to $N$. Now let $i_2$ be the smallest index in $\{1,..., m\}$ so that $e_{i_2}$ is not in the $\Q$-span of $N$ and $e_{i_1}$, add $e_{i_2}$ to the matrix. Continue in this way to arrive at a (necessarily square) matrix $N'$. We check that $\coker(N)_T \leq \coker(N')_T$. \\

We check the subgroup inclusion inductively. Suppose $v$ is a torsion element of the span of the columns of $N$ together with standard basis vector $e_{i_1}, e_{i_2}, \cdots, e_{i_k}$, but $v$ is not a torsion element of the cokernel after adding the column $e_{i_{k + 1}}$. Then $v$ can be written as an integral linear combination of columns of $N$ and standard basis vectors $e_{i_1}, e_{i_2}, \cdots, e_{i_k}, e_{i_{k + 1}}$, with nonzero coefficient $\alpha$ on $e_{i_{k + 1}}$. However since $v$ is a torsion element of the cokernel before adding $e_{i_{k + 1}}$ we have that there exists an nonzero integer $s$ so that $sv$ is in the integer span of the columns $N$ and $e_{i_1}, e_{i_2}, \cdots, e_{i_k}$. This gives us two ways to write $sv$ as a linear combination of columns of $N$ and $e_{i_1}, e_{i_2}, \cdots, e_{i_k}, e_{i_{k + 1}}$, one with coefficient $s\alpha$ on $e_{i_{k + 1}}$ and one with coefficient 0 on $e_{i_{k + 1}}$. By linear independence of the columns of $N'$ we have that $s \alpha = 0$, a contradiction. Thus $\coker(N)_T \leq \coker(N')_T$ and we complete the proof.
\end{proof}
Similar bounds appear in \cite{HKP2, Kalai, LP2}. In \cite{Kalai}, it is shown that if $X$ is a $d$-dimensional simplicial complex on $n$ vertices then $|H_{d - 1}(X)_T| \leq \sqrt{d + 1}^{\binom{n - 2}{d}}$. Claim \ref{matrixbound} essentially gives a local version of this result.\\

\begin{lemma}\label{probabilitybound}
  For any $c > (d-1/2),$
  there is an $n_0$ sufficiently large
  so that
  for all $n \geq n_0,$
  for all $(d - 1)$-dimensional complexes $X$ with $b(X) = (1 - \theta)nk \geq nk/(d + 1),$ where $k := |X|$, 
  and for $Y \sim Y_d(n, \lceil \frac{c \log n}{n} \binom{n}{d + 1} \rceil)$, 
  the probability that $z(X)$ holds is at most $n^{-(1 - \theta) (d - 1/2)k}.$
\end{lemma}
\begin{proof}
  Fix a $(d - 1)$-dimensional complex with $b(X) = (1 - \theta) nk$ where $k := |X|$, and fix a field $R$ to be either $\Z/q\Z$ for $q \leq \sqrt{d + 1}^{k}$ or to be $\Q$. We will bound the probability of $z(X, R)$ for $R$ fixed and then take a union bound over all at-most $\sqrt{d + 1}^{k}$ necessary fields to bound $z(X)$. We will use the Linial--Meshulam stochastic process $\Y_d(n) = \{Y_d(n,i)\}_{i=0}^{\binom{n}{d+1}}$ to sample from $Y_d(n,m)$ where $m := \lceil \frac{c \log n}{n}\binom{n}{d + 1} \rceil$.  For each $i$, let $z(X, R, i)$ be the event that $X$ is the support of a cocycle over $R$ of weight $k$ in $Y_d(n, i)$. At each step $i,$ we let $\X(i)$ denote the dimension of the kernel of the coboundary matrix of $Y_d(n, i)$ restricted to the columns associated to $X$. Clearly if $\X(m) = 0$, then $X$ is not the support of a cocycle over $R$.\\

For each $i$ let $p_i$ denote the probability that $\X(i + 1) < \X(i)$. Now if $X$ is the support of a cocycle of weight $k$ over $R$ in $Y_d(n, i)$ then 
$$p_i \geq \frac{b(X)}{\binom{n}{d + 1}} \geq \frac{(1 - \theta)nk}{\binom{n}{d + 1}}.$$

We want to bound the following probability
$$\Pr((\X(m) > 0) \cap z(X, R, m)).$$

Note that $z(X, R, m)$ implies $\X(m) > 0$ so the probability above is actually equal to $\Pr(z(X, R))$. Clearly,
\begin{eqnarray*}
\Pr(\X(m) > 0 \cap z(X, R, m)) &\leq& \Pr\left((\X(m) > 0) \cap z(X, R, m) \mid \forall i \leq m, p_i \geq\frac{(1 - \theta)nk}{\binom{n}{d + 1}} \right) \\
&&+  \Pr\left((\X(m) > 0) \cap z(X, R, m) \mid \exists i \leq m, p_i <\frac{(1 - \theta)nk}{\binom{n}{d + 1}} \right) \\
&\leq& \Pr \left( \X(m) > 0 \mid \forall i \leq m, p_i \geq\frac{(1 - \theta)nk}{\binom{n}{d + 1}}\right) \\
&& + \Pr\left(z(X, R, m) \mid \exists i \leq m, p_i <\frac{(1 - \theta)nk}{\binom{n}{d + 1}} \right).
\end{eqnarray*}
Now the second summand is zero. Indeed while $z(X, R, i)$ holds $p_i \geq \frac{(1 - \theta)nk}{\binom{n}{d + 1}}$ and if $z(X, R, i)$ fails for some $i \leq m$, then so does $z(X, R, m)$. The goal is to bound the first summand. This will be accomplished by comparison to a binomial random variable. Let $B$ be a binomial random variable with $m$ trials and success probability $\psi = \dfrac{(1 - \theta) nk}{\binom{n}{d + 1}}$. Since $\X(0) = k$, it follows that 
$$\Pr\left(\X(m) > 0 \mid \forall i \leq m, p_i \geq \frac{(1 - \theta) nk}{\binom{n}{d + 1}} \right) \leq \Pr(B < k).$$
We bound the probability that $B$ is less than $k,$ which since $k$ is less than $\mathbb{E}(B) = m\psi = \Omega(k \log n),$ allows the following version of Chernoff's inequality to apply: 
\[
  \Pr(B < k) \leq \exp( mH_\psi(k/m)),
\]
where $H_\psi(x) = x\log(\psi/x) + (1-x)\log( (1-\psi)/(1-x)).$

Observe that $k$ is no more than $n^d,$ and that $n^d/m = O(1/\log n).$
Hence uniformly in $k \leq n^d,$ we have that
\[
  \begin{aligned}
    mH_\psi(k/m)
    &=
    k\log(m\psi/k)
    +(m-k)\log(1-\psi)
    -(m-k)\log(1-k/m)  \\
    &\leq
    k\log(m\psi/k)
    -m\psi + k(1+\psi) + O(k^2/m) \\
    &\leq
    -m\psi
    +k\log\log n
    +k(1+O(1/\log n)).
  \end{aligned}
\]
Thus there is a constant $C>0$ so that for all $n$ large enough and for all $k \leq n^d$
\begin{eqnarray*}
  \Pr(B < k) \leq n^{-(1 - \theta)ck}(C\log n)^{k}.
\end{eqnarray*}
%Thus $\Pr(z(X, R)) \leq n^{-((1 - \theta)c + \epsilon)k}$ for any arbitrary $R$ provided that $n$ is larger than some absolute constant, which does not depend on $R$. 
Now we sum over all fields $\Z/q\Z$ with $q \leq \sqrt{d + 1}^k$ and the field $\Q$ to get that the probability that $X$ is the support of a cocycle over any such field is at most:
\begin{equation} \label{eq:truebound}
  \Pr(z(X)) \leq n^{-(1 - \theta)ck}(C\sqrt{d+1}\cdot \log n)^{k}.
\end{equation}
This gives the desired bound for all $n$ sufficiently large.
\end{proof}

\begin{remark}
We should point out here that there is nothing particularly meaningful about the choice of $(d - 1/2)$ in the lemma. Later, this turns out to be a convenient value to have, and so we use it here to simplify some notation. 
%However, what we really prove in proving Lemma \ref{probabilitybound} is that if $p = \frac{c \log n}{n}$, for any $c$, and $X$ is a $(d - 1)$-dimensional complex contained in the simplex on $n$ vertices then the probability that $z(X)$ holds in $Y_d(n, p)$ is at most $\left(1 - \dfrac{(c - \epsilon) \log n}{n} \right)^{b(X)}$ for any $\epsilon > 0$ and $n$ large enough. 

Also, recalling the definition of $b(X),$ there is some cochain $\phi$ over some field $R$ such that $\phi$ is supported on $X$ and so that $b(\phi)=b(X).$  For this particular cochain, the probability it is a cocycle is $(1 - p)^{b(X)}= n^{-(1 - \theta)ck}$.  Hence we have
\[
 n^{-(1 - \theta)ck} \leq \Pr(z(X)) \leq n^{-(1 - \theta)ck}(C\sqrt{d+1}\cdot \log n)^{k},
\]
showing that the above bound is accurate up to subleading factors.
\end{remark}

\section{Overview of the proof}
Now that we have defined $z(X)$, we can give an outline of the proof of Theorem \ref{mainresult}. Essentially the idea of the proof will be to prove that with high probability $Y \sim Y_d(n, \frac{c \log n}{n})$ satisfies three particular conditions for $c > d - 1/2$ and then to show that these three conditions deterministically imply Theorem \ref{mainresult}. The majority of the work of this paper is to prove the former, which we state below as Lemma \ref{keylemma}. To simplify notation, for a $d$-dimensional simplicial complex, we use ``face", ``facet", and ``ridge" to refer to $d$-dimensional faces, $(d - 1)$-dimensional faces, and $(d - 2)$-dimensional faces respectively. We also say that a cocycle $\phi$ is inclusion-minimal over a field $R$ if there is no cocycle over $R$ supported on any proper subset of the support of $\phi$. 
\begin{lemma}\label{keylemma}
Fix $d \geq 2$ and $c > d - 1/2$, then with high probability $Y \sim Y_d(n, \frac{c \log n}{n})$ satisfies the following three conditions:
\begin{enumerate}
\item $z(X)$ fails to hold for all $(d - 1)$-subcomplexes $X$ with $|X| \geq n/(3d)$
\item $Y$ has no inclusion minimal $(d-1)$-cocycles of support size $k$ over any field for $2 \leq k \leq n/(3d)$.
\item $Y$ has no isolated facets that meet at a ridge.
\end{enumerate}
\end{lemma}

These three conditions, in turn, imply the desired homology vanishing on $Y$, as a consequence of the following \emph{deterministic} lemma.

\begin{lemma}\label{deterministic}
%Suppose $Y$ is a $d$-dimensional simplicial complex with complete $(d - 1)$-skeleton for which the three conditions of Lemma \ref{keylemma} hold. Then the stochastic process of adding $d$-diimensional faces uniformly at random to $Y$ will result in a complex $Y' \supset Y$ which has $H_{d - 1}(Y') = 0$ exactly at the moment the final isolated facet of $Y$ is covered. 

Suppose that $Y$ is a $d$-dimensional simplicial complex with complete $(d - 1)$-skeleton so that conditions 1 and 2 from Lemma \ref{keylemma} hold, then $H_{d - 1}(Y)$ is a free abelian group of rank equal to the number of isolated facets of $Y$. Moreover if all three conditions hold then the stochastic process of adding $d$-dimensional faces uniformly at random to $Y$ will result in a complex $Y' \supset Y$ which has $H_{d - 1}(Y') = 0$ exactly at the moment the final isolated facet of $Y$ is covered. 
\end{lemma}
\noindent In the proof of this lemma we encounter the term \emph{strongly-connected} which we define here and will use again in Section \ref{smallcocycles}.
\begin{definition}
  For a $d$-dimensional simplicial complex $X$ we define the \emph{dual graph} $G(X)$ to be the graph whose vertex set is the set $d$-dimensional faces of $X$ with an edge between $\sigma$ and $\tau$ provided that $\sigma$ and $\tau$ intersect at a $(d - 1)$-dimensional face. We say that $X$ is \emph{strongly-connected} if its dual graph is connected. Equivalently, $X$ is strongly connected if for every $\sigma$ and $\tau$ there is a path $\sigma = \sigma_1, \sigma_2, \sigma_3, \cdots, \sigma_k = \tau$, so that for every $1 \leq i \leq k - 1$, $|\sigma_i \cap \sigma_{i + 1}| = d - 1$. 
\end{definition}

\begin{proof}[Proof of Lemma~\ref{deterministic}]
First suppose that $Y$ is such that conditions 1 and 2 hold. Then we have that over $\Q$ every inclusion-minimal cocycle of $Y$ is an isolated facet. Thus $H^{d - 1}(Y; \Q)$ is generated by isolated facets of $Y$. It follows that the same holds for $H^{d - 1}(Y; \Z)$ (recall that torsion subgroups ``shift up" one dimension when we change from homology to cohomology, so we don't immediately have that $H_{d - 1}(Y; \Z)$ is free). We claim that $\beta^{d - 1}(Y; \Z/q\Z) = \beta^{d - 1}(Y; \Q)$ for every prime $q$. This will imply that there is no torsion in homology and so $H^{d - 1}(Y; \Z)$ will be isomorphic to $H_{d - 1}(Y; \Z)$ proving the first part of the claim. Suppose there is a prime $q$ so that $\beta^{d - 1}(Y; \Z/q\Z) > \beta^{d - 1}(Y; \Q)$. Then $Y$ has a nontrivial cocycle $\phi$ with coefficients in $\Z/q\Z$ that is not the image of an integral cocycle modulo $q$. Let $X$ be the support of a minimal-weight representative of $\phi$. We may assume that $X$ has no isolated facets (otherwise we could subtract from $\phi$ an appropriate multiple of a cochain supported on exactly an isolated facet of $X$ to arrive at a new cocycle over $\Z/q\Z$ which is not the image of a cocycle over $\Z$ and has smaller support). \\

Now over $\Z$ we have $\partial_d^* |_X(\phi) = q \psi$ for some integral vector $\psi$. Moreover by conditions 1 and 2, $q > \sqrt{d + 1}^{|X|}$ (and $|X| \geq n/(3d)$). But this $q$ is too large relative to $|X|$ for the cokernel of $\partial_d^*|_X$ to have $q$-torsion by Claim \ref{matrixbound}, thus $q \psi$ in the image of $\partial_d^*|_X$ over $\Z$ implies that $\psi$ is also in the image over $\Z$. Thus there exists an integral vector $\phi'$ so that $\phi - q \phi'$ is supported on $X$ and is a cocycle over $\Z$. However since $X$ has no isolated facets and $H^{d - 1}(X, \Z)$ is generated by isolated facets we have that $\phi - q \phi'$ is a coboundary over $\Z$, but then modding out by $q$ gives us that $\phi$ is a coboundary over $\Z/q\Z$ contradicting our assumption on $\phi$. This proves the first part of the claim. \\

Now suppose that $Y$ satisfies conditions 1, 2, and 3 and let $Y'$ be the complex at the moment in the stochastic process where the last isolated facet of $Y$ is covered. We wish to show that $H_{d - 1}(Y') = 0$. We will prove this by induction on the number of isolated facets of $Y$. By conditions 1 and 2, $H_{d - 1}(Y)$ is a free abelian group generated by the isolated facets of $Y$. Thus if $Y$ has no isolated facets then $H_{d - 1}(Y) = 0$ and $Y' = Y$ so we have the result. For the inductive step, we prove that if $Y$ satisfies conditions 1, 2, and 3 then for any face $\sigma$ which could be added to $Y$, we have that $Y \cup \{\sigma\}$ still satisfies 1, 2, and 3. This will prove that 1, 2, and 3 hold at every step and eventually we cover some isolated facet of the complex and then can apply induction. \\

Conditions 1 and 3 are clearly monotone, so we only have to show that condition 2 is monotone under our other assumptions. Suppose not. Let $\sigma$ be a face so that $Y$ satisfies conditions 1, 2, and 3 but $Y \cup \{\sigma\}$ does not satisfy condition 2. Let $\phi$ be an inclusion-minimal cocycle, with weight and support size at least two, over some field $R$ for $Y \cup \{\sigma\}$. Since $\phi$ is a cocycle of $Y \cup \{\sigma \}$ it is a cocycle for $Y$. But because $Y$ satisfies 1 and 2, we have that $H^{d - 1}(Y; R)$ is generated by isolated facets. Therefore the support of $\phi$ is a union of isolated facets. By inclusion-minimality (after adding $\sigma$) the support of $\phi$ is strongly connected too. However, by 3 we have that the support of $\phi$ must be a single isolated facet, so $\phi$ has support size one, contradicting our assumption. This shows that $Y \cup \{\sigma\}$ satisfies condition 2 and we finish the proof by induction.
\end{proof}

Before we get to the proof of Theorem \ref{mainresult} we state the following corollary which characterizes the structure of $H_{d - 1}(Y)$ for $Y \sim Y_d(n, \frac{c \log n}{n})$ and $c > d - 1/2$. The proof is immediate from Lemmas \ref{keylemma} and \ref{deterministic}:
\begin{corollary}\label{isolatedfaces}
If $c > d - 1/2$, then with high probability $H_{d - 1}(Y)$ for $Y \sim Y_d(n, \frac{c \log n}{n})$ is a free abeilan group of rank equal to the number of isolated facets of $Y$. 
\end{corollary}
Together with a first moment argument showing that $Y \sim Y_d(n, \frac{c \log n}{n})$ has no isolated facets with high probability for $c > d,$
this corollary establishes the sharp threshold for integral homology to vanish in $Y_d(n, p)$. Now we give the proof of the main result.
\begin{proof}[Proof of Theorem \ref{mainresult}]
Consider an instance of $\Y_d(n)$, and let $m_0$ be the hitting time for the event that the final isolated facet of $\Y_d(n)$ is covered. Clearly $Y_d(n, i)$ has nontrivial $(d - 1)$st homology group for $i < m_0$. That is the hitting time for $(d-1)$st homology to vanish is not earlier than the hitting time for the final isolated facet to be covered. It therefore suffices to show that with high probability $H_{d - 1}(Y_d(n, m_0)) = 0$. \\

First generate $Y \sim Y_d \left( n, \frac{(d - 1/4) \log n}{n} \right)$. If $Y$ has isolated facets, then run the stochastic Linial--Meshulam process starting at $Y$ and continuing until the moment the last isolated facets is covered. In the case that $Y$ has isolated facets, this generates a complex $Y_d(n, m_0)$ in $\Y_d(n)$. By Lemma \ref{keylemma}, $Y$ satisfies the three stated conditions with high probability and so by Lemma \ref{deterministic}, the probability that $Y_d(n, m_0)$ has nontrivial $(d - 1)$st homology group given that $Y$ has isolated facets is $o(1)$. Thus it suffices to check that the probability that $Y$ has no isolated facets is also $o(1)$, but this follows from a straightforward second moment argument.
\end{proof}

The rest of the paper will be devoted to proving Lemma \ref{keylemma}. Condition 1 will be referred to as the ``large cocycle" condition and will be proved in Section \ref{largecocycles}. Condition 2 will be referred to as the ``small cocycle" condition and will be proved in Section \ref{smallcocycles}. Condition 3 is much easier and we prove it now

\begin{lemma}\label{noadjacent}
If $Y \sim Y_d(n, c \log n/n)$ and $c > (d + 1)/2$ then with high probability $Y$ does not contain two isolated facets that meet at a ridge.
\end{lemma}
\begin{proof}
We will use the first moment method. Two isolated facets that meet a ridge is a subcomplex with two $(d - 1)$-dimensional faces, which are both isolated, and $d + 1$ vertices. The number of such complexes is at most $\binom{n}{d + 1} \binom{d + 1}{d-1}$. The probability that both facets are isolated is at most $(1 - p)^{2(n - d) - 1}$. Thus the expected number of pairs of isolated facets that meet at a ridge is at most
\begin{eqnarray*}
\binom{n}{d + 1} \binom{d + 1}{d-1}(1 - p)^{2(n - d) - 1} &\leq& n^{d + 1} (d + 1)^2 \exp \left( -\frac{c \log n}{n} (2(n - d) - 1) \right) \\
&\leq& n^{d + 1}(d + 1)^2 n^{-2c(1 - o(1)) }
\end{eqnarray*}
This is $o(1)$ since $2c > (d + 1)$. 
\end{proof}

\section{Large cocycles}\label{largecocycles}
The goal of this section is to prove the following lemma about the large cocycle condition. This will be accomplished by using Lemma \ref{probabilitybound} together with an enumeration result from \cite{MW} to bound the probability that $Y \sim Y_d(n, \frac{c \log n}{n})$ contains a $(d - 1)$-dimensional subcomplex $X$, with $|X| \geq n/(2d)$, for which $z(X)$ holds.
\begin{lemma}\label{nolargecocycles}
If $Y \sim Y_d(n,  \frac{c \log n}{n})$ and $c > d - \frac{1}{2}$ then with high probability $z(X)$ fails to hold for all $(d - 1)$-dimensional complexes on $n$ vertices of size at least $n/(3d)$
\end{lemma}

Similar to the approach in \cite{MW}, but now avoiding having to deal with coefficients, we want to count the number of $(d - 1)$-complexes on $n$ vertices with $b(X) = (1 - \theta) n |X|$. To do so we recall that $\beta(X) \leq b(X)$, and we count the number of complexes with $\beta(X) \leq (1 - \theta)n|X|$. We make use of the following lemma from Meshulam--Wallach.

\begin{lemma}[Claim 4.2 from \cite{MW}]\label{claim42}
Let $0 < \epsilon \leq 1/2$ and then for $n$ large enough and $X$ so that $\beta(X) \leq (1 - \theta)|X|(n - d)$ for some $0  < \theta \leq 1$, there exists a subfamily $S \subseteq X$ of size less than $C \frac{|X|}{n} + 2\log\frac{1}{\epsilon\theta}$ such that $\Gamma(S) := \{ \tau \in X : |\tau \cap \sigma| = d - 1 \text{ for some $\sigma \in S$}\}$ has size at least $(1 - \epsilon)\theta |X|$, where $C$ is a constant depending only on $\epsilon$ and $d$.
\end{lemma}

We are now ready to use Lemma \ref{claim42} to count the number of complexes $X$ with $\beta(X) \leq (1 - \theta)|X|n$. This will upper bound the number of complexes $X$ with $b(X) = (1 - \theta)|X|n$.

\begin{lemma}[Modification to Proposition 4.1 from \cite{MW}]\label{betabound}
For $n$ large enough, $k \geq n/(3d)$, and $\theta \geq 1/(2d)$ there exists a constant $c = c(d)$ so that the number of $(d - 1)$-complexes $X$ with $|X| = k$ and $\beta(X) \leq (1 - \theta)kn$ is at most 
$$\left( c n^{(d - 1)(1 - \theta(1 - \frac{1}{2d^2}))} \right)^k$$
\end{lemma}

We give the proof here, essentially as it appears in \cite{MW}, though we omit any consideration of an underlying coefficient ring. The proof follows directly from Lemma \ref{claim42}. 

\begin{proof}
For $n$, $k$, and $\theta$, let $\mathcal{F}_n(k, \theta)$ denote the collection of $(d - 1)$-dimensional subcomplexes $X$ of the simplex on $n$ vertices with $k$ facets and $\beta(X) \leq (1 - \theta)kn$.
If $X \in \mathcal{F}_n(k, \theta)$ then $$\beta(X) \leq \left(1 - \frac{\theta n - d}{n - d} \right) k (n - d)$$

Suppose that $\theta \geq 1/(2d)$ and let $\theta' = \dfrac{\theta n - d}{n - d}$ and $\epsilon = 1/(2d^2)$, by Lemma \ref{claim42} when $n$ is large enough we obtain for every $X \in \mathcal{F}_n(k, \theta')$ where $k \geq n/(3d)$ a set $S$ of size at most $C\frac{k}{n}$, for some constant $C$ depending on $\epsilon$,\footnote{This is where we use that fact that $k \geq n/3d$ and $\theta \geq 1/(2d)$. Indeed the $C$ from Lemma \ref{claim42} gives us a bound of $C \frac{|X|}{n} + 2\log \frac{1}{\epsilon \theta} \leq C \frac{|X|}{n} + 2 \log (4d^3) \leq C \frac{|X|}{n} + 6d \log (4d^3) \frac{|X|}{n}$. So the $C$ in this proof should be the $C$ in Lemma \ref{claim42} plus $6d \log (4d^3)$. Of course we could set any $\delta, \theta_0 >0$ and assume that $k \geq \delta n$ and $\theta \geq \theta_0$, but the choices of $1/(3d)$ and $1/(2d)$ respectively are convenient in other parts of our paper.} so that $\Gamma(S)$ has size at least $(1 - 1/(2d^2)) \theta' k$. Thus we get a map taking $X$ in $\mathcal{F}_n(k, \theta)$ to $(S, \Gamma(S), X - \Gamma(S))$. Since the latter two coordinates of this 3-tuple give a partition of $X$, this map is injective, so the cardinality of $\mathcal{F}_n(k, \theta)$ is at most the number of such tuples. Therefore it is at most
\begin{eqnarray*}
\left(\sum_{i = 0}^{Ck/n} \binom{\binom{n}{d}}{i} \right) \left( 2^{\frac{Ck}{n} dn} \right) \left( \sum_{j = 0}^{k - \theta'k(1 - 1/(2d^2))} \binom{\binom{n}{d}}{j}\right)
\end{eqnarray*}
Now the first two factors in the product above are at most $c_1^k$ and $c_2^k$ respectively for some constants $c_1$ and $c_2$ depending on $d$. Thus for $n$ large enough and $k \leq \binom{n}{d}/2$, we have
\begin{eqnarray*}
%(c_1c_2)^k k (1 - \theta'(1 - 1/(2d^2))) \binom{n^d}{k (1 - \theta'(1 - 1/(2d^2)))} &\leq& \left(\frac{2c_1c_2e}{1 - \frac{1}{4d}(1 - \frac{1}{2d^2})}\right)^k \left( \frac{n^d}{k} \right)^{(1 - \theta'(1 - 1/(2d^2))k}
|\mathcal{F}_n(k, \theta)| \leq (c_1c_2)^k k (1 - \theta'(1 - 1/(2d^2))) \binom{n^d}{k (1 - \theta'(1 - 1/(2d^2)))}
\end{eqnarray*}
Therefore there exists a constant $c$ so that for $n/(3d) \leq k \leq \binom{n}{d}/2$, 
\begin{eqnarray*}
|\mathcal{F}_n(k, \theta)| \leq c^k n^{(d - 1) (1 - \theta (1 - \frac{1}{2d^2}))k}
\end{eqnarray*}
This finishes the proof of the lemma in the case that $k \leq \binom{n}{d}/2$. In the case that $k$ is larger than $\binom{n}{d}/2$, we may use the trivial bound of $2^{\binom{n}{d}}$ on $|\mathcal{F}_n(k, \theta)|$ and so there is nothing to prove. 
\end{proof}

We are now ready to combine Lemma \ref{betabound} with Lemma \ref{probabilitybound} to show that with high probability $z(X)$ fails to hold for every $(d- 1)$-dimensional complex $X$ on $n$ vertices with at least $n/(3d)$ facets.
\begin{lemma}\label{noZ}
If $Y \sim Y_d(n, \frac{c \log n}{n} \binom{n}{d + 1})$ for $c > d - 1/2$ then with high probability $z(X)$ fails to hold for all $(d - 1)$-dimensional complexes on $n$ vertices of size at least $n/(3d)$.
\end{lemma}
\begin{proof}
For any $(d - 1)$-dimensional subcomplex $X$, we have that the probability of $z(X)$ is at most $n^{-(1 - \theta)(d - 1/2)k}$ where $b(X) = (1 - \theta)nk$ by Lemma \ref{probabilitybound}. Now if we define $f_n(k, \theta)$ to be the number of $X$ with $|X| = k$ and $b(X) = (1 - \theta)nk$, then we wish to show that 
$$\sum_{k \geq n/(3d)} \sum_{\theta \in \Upsilon} f_n(k, \theta) n^{-(1 - \theta)(d - 1/2)k} = o(1),$$
where we set $\Upsilon = \{\theta : (1 - \theta)nk \in \Z, (1 - \theta)nk \leq n^{d + 1}, \theta \leq d/(d + 1)\}.$

For $\theta \in \Upsilon_1 := \{ \theta \in \Upsilon : \theta < 1/(2d)\}$, we cannot apply Lemma \ref{betabound}, but the trivial bound on $\sum_{\theta} f(k, \theta) \leq \binom{n^d}{k}$ works instead. Indeed we have,
\begin{eqnarray*}
\sum_{k \geq n/(3d)} \sum_{\theta \in \Upsilon_1} f(k, \theta) n^{-(1 - \theta)(d - 1/2)k} &\leq& \sum_{k \geq n/(3d)} n^{(d-1)k} n^{-(1 - 1/(2d))(d - 1/2)k}\\
&\leq& \sum_{k \geq n/(3d)} n^{-k/(4d)} = o(1).
\end{eqnarray*}

For $\theta \in \Upsilon_2 := \{ \theta \in \Upsilon : \theta \geq 1/(2d)\}$ we will apply Lemma \ref{betabound}, as $b(X) = (1 - \theta)nk$ implies that $\beta(X) \leq (1 - \theta)nk$ so $f_n(k, \theta) \leq |\mathcal{F}_n(k, \theta)|$. We have
\begin{eqnarray*}
\sum_{k \geq n/(3d)} 
\sum_{\theta \in \Upsilon_2} 
f(k, \theta) n^{-(1 - \theta)(d - 1/2)k} 
&\leq& 
\sum_{k = n/(3d)}^{n^d} 
\sum_{\theta \in \Upsilon_2} 
\left( c n^{(d - 1)(1 - \theta(1 - \frac{1}{2d^2}))} \right)^k n^{-(1 - \theta)(d - 1/2)k} \\
&\leq& 
\sum_{k = n/(3d)}^{n^d} 
\sum_{\theta \in \Upsilon_2} 
\left(c n^{-\frac{1}{2} + (\frac{1}{2} + \frac{1}{2d} - \frac{1}{2d^2})\theta} \right)^k \\ 
&\leq& 
\sum_{k = n/(3d)}^{n^d} 
\sum_{\theta \in \Upsilon_2} 
\left(c n^{-\frac{1}{2} + (\frac{1}{2} + \frac{1}{2d} - \frac{1}{2d^2})\frac{d}{d + 1}} \right)^k \\
&\leq& 
\sum_{k = n/(3d)}^{n^d} 
\sum_{\theta \in \Upsilon_2} 
\left(c n^{-\frac{1}{2d(d + 1)}} \right)^k \\
&\leq& n^{d + 1} 
\sum_{k = n/(3d)}^{n^d} 
\left(c n^{-\frac{1}{2d(d + 1)}} \right)^k 
\leq n^{2d + 1}n^{-\Theta(n)} 
= o(1).
\end{eqnarray*}
\end{proof}

Finally, we use a simple coupling argument to finish the proof of Lemma \ref{nolargecocycles}.
\begin{proof}[Proof of Lemma \ref{nolargecocycles}]
Fix $c' \in (d - 1/2, c)$. Let $Y_1 \sim Y_d(n, \frac{c' \log n}{n} \binom{n}{d + 1})$. With high probability $Y_1 \subseteq Y$. Indeed, by a routine application of Chernoff's bound, if $Y$ is distributed as $Y_d(n, \frac{c \log n}{n})$ then the probability that $Y$ has fewer than $\frac{c' \log n}{n} \binom{n}{d + 1}$ faces is at most 
$$\exp \left( - \frac{ \frac{\log n}{n} \binom{n}{d + 1} (c - c')^2}{2c} \right) = o(1).$$
And by Lemma \ref{noZ}, with high probability $z(X)$ fails to hold for all $X$ on $n$ vertices of size at least $n/(2d)$ in $Y_1$, and hence the same holds in $Y$ since the kernel of the $d$th coboundary map of $Y$ is contained in the kernel of the $d$th coboundary map of $Y_1$.
\end{proof}

\section{Small cocycles}\label{smallcocycles}

To show that the small cocycle condition holds for $Y \sim Y_d(n, \frac{c \log n}{n})$ with high probability for $c > d - 1/2$ holds we rely on the fact that the support of an inclusion minimal cocycle is strongly-connected.
Strongly connected $(d - 1)$-complexes with $k$ facets on $n$ vertices 
are relatively few, in comparison to those that are not strongly connected.
Specifically, we have the following:
\begin{lemma}\label{stronglyconnected}
  The number of strongly-connected $(d - 1)$-complexes with $k$ facets on $n$ vertices is at most
  \[
    n^{d+k-1} (2d)^{k}.
  \]
\end{lemma}
\begin{proof}
  A complex $X$ is strongly connected if and only if the dual graph of $G(X)$ is a connected subgraph of the dual graph $H$ of the $(d-1)$--skeleton of the full simplex on $n$ vertices.  As each $G(X)$ has a spanning tree, we estimate the number of $X$ above by the number of rooted subtrees of $H$ with $k$ vertices.  There are at most
\[
  \binom{n}{d} \cdot 2^{k-1} \cdot (dn)^{k-1}
\]
such rooted subtrees, enumerated in breadth--first--search order.  In this enumeration, the $\binom{n}{d}$ counts the number of choices for the root, the $2^{k-1}$ counts the ways to partition the remaining $k-1$ into the sizes of the neighborhoods of each vertex in the breadh--first--search, and the $(dn)^{k-1}$ overestimates the number of ways to pick the neighborhoods.
\end{proof}

%\begin{remark}
%  In \cite[Paragraph above Lemma 1]{ALLM} it is suggested that the bound $n^{d - 1 + k}$ is easily derived, without proof.
%\end{remark}
%
\begin{lemma}
If $c > (d + 1)/2$ then with high probability $Y \sim Y_d(n, c \log n / n)$ has no inclusion-minimal $(d-1)$-cocycles of support size $k$ over any field for $2 \leq k \leq \log n$.
\end{lemma}

\begin{proof}
  We use the first moment method. If $X$ is the support of an inclusion-minimal cocycle of $Y$ over some field then $X$ is a strongly-connected $(d - 1)$ complex. Moreover, if $|X| = k \leq \log n$ then at least $n - d \log n$ vertices do not belong to $X$. Now if $X$ is to be the support of a cocycle over any field, then any face $\tau$ obtained as the union of a $(d - 1)$-dimensional face of $X$ and a vertex outside of $X$ must be excluded from $Y$. Thus the probability that a fixed $X$ of size $k \leq \log n$ is the support of a cocycle over any field is at most $(1 - p)^{k(n - d \log n)}$. 
  
  Since $X$ must be strongly connected, the number of choices for $X$ with size $k$ is at most $(2d)^k n^{d - 1 + k}$ by Lemma~\ref{stronglyconnected}. Applying the union bound over $k \in \{2, 3, ..., \log{n}\}$ for the probability that there exists an inclusion-minimal cocycle of support size $k$ we have:
\begin{eqnarray*}
\sum_{k = 2}^{\log{n}} (2d)^kn^{d - 1 + k} (1 - p)^{kn(1 - o(1))} &\leq& n^{d - 1} \sum_{k = 2}^{\log{n}}(2d)^k n^k n^{-ck(1 - o(1))} \\
&\leq& n^{d - 1} (\log{n}) n^{2 - 2c + o(1)} \\
&=& (\log{n}) n^{d - 1 - 2(c - 1) + o(1)} = o(1).
\end{eqnarray*}
\end{proof}

\begin{lemma}
If $c > 3/2$, then with high probability $Y \sim Y_d(n, c \log n /n)$ has no inclusion minimal $(d - 1)$-cocycles of support size $k$ over any field for $\log{n} \leq k \leq n/(3d)$.
\end{lemma}
\begin{proof}
If $X$ is the support of such a cocycle then there are at least $n - n/3 = 2n/3$ vertices of $Y$ outside of $X$. Taking this consideration and the same argument as the proof of $k \leq \log{n}$, we bound the following to prove the lemma:
\begin{eqnarray*}
\sum_{k = \log{n}}^{n/(3d)} n^{d - 1 + k} (1 - p)^{2kn/3} &\leq& n^{d - 1} \sum_{k = \log{n}}^{n/(3d)} n^k n^{-2ck/3} \\
&\leq& n^d (n^{2c/3 - 1})^{-\log{n}}\\
&=& n^{-\Theta(\log{n})}.
\end{eqnarray*}
\end{proof}
This finishes the proof of Lemma \ref{keylemma}, and hence the proof of our main result. 

\section{Conclusion}
Our result finally establishes $p = \frac{d \log n}{n}$ as the sharp threshold for homological connectivity of $Y_d(n, p)$. Moreover, Corollary \ref{isolatedfaces} tells us about the structure of the $(d - 1)$st homology group immediately before it vanishes. However, the following two questions are closely related to our main result and remain open.
\begin{itemize}
\item What is the homological connectivity threshold for the random hypergraph model? This model is similar to the Linial--Meshulam model except that one does not start with the complete $(d - 1)$-skeleton. Rather the $d$-dimensional faces are included independently and the complex is obtained by taking the downward closure of the top-dimensional faces. In \cite{CdGKS}, Cooley et al.\ show the hitting-time result for homological connectivity with $\Z/2\Z$ coefficients in the random hypergraph model. Their result establishes that the sharp threshold for homological connectivity with $\Z/2\Z$-coefficients for the random hypergraph model is at $\dfrac{d \log n}{2n}$. Can our methods be adapted to obtain the corresponding result with integer coefficients? 

\item The question about torsion in homology of $Y_d(n, p)$ is raised in \cite{KLNP, LP2}. Namely, experimental evidence strongly suggests that shortly before the first nontrivial cycle\footnote{That is the first top homology class not generated by boundaries of the $(d + 1)$-simplex} appears in the \emph{top} homology group of $\Y_d(n)$, there is an exceptionally large (on the order of $\exp(\Theta(n^d))$) torsion group which appears in the $(d - 1)$st homology group. Outside of this however, it is believed that $\Y_d(n)$ has no torsion in homology; \cite{LP2} formulates this conjecture precisely. Our paper in fact grew out of an attempt to prove the stronger result that for $c$ a sufficiently large constant and $p = c/n$, one has that with high probablity the $(d - 1)$st homology group of $Y_d(n, p)$ is torsion free. However this problem remains open.
\end{itemize}

\bibliography{ResearchBibliography}
\bibliographystyle{amsplain}
\end{document}